\documentclass[english,12pt]{article}

\pdfoutput=1

\usepackage[T1]{fontenc}
\usepackage[latin9]{inputenc}
\usepackage{amssymb}
\usepackage{babel}
\usepackage{amsthm}
\usepackage{verbatim}
\usepackage{amsrefs}
\usepackage{amsmath}
\usepackage{listings}
\usepackage{fullpage}
\usepackage{graphicx}
\usepackage{qtree}
\usepackage[usenames,dvipsnames]{color}
\usepackage{array}
\usepackage{enumerate}
\usepackage{polynom}
\usepackage{tikz}
\usepackage{fancyhdr}
\usepackage{etoolbox}
\usepackage{algorithm}
\usepackage{algorithmic}

\newtheorem{defin}{Definition}

\newtheorem{theorem}{Theorem}
\newtheorem{lemma}{Lemma}

\newcommand{\p}[1]{\left(#1\right)}

\newcommand{\bk}[1]{\left[#1\right]}

\newcommand{\quot}[1]{``#1''}

\DeclareMathOperator{\modd}{mod}

\newcommand{\namehead}[3]{
\lstset{breaklines=true, morecomment=[l]{//}, frame=single, showstringspaces=false, numbers=left}
\begin{flushright}
Nathan Fox\\
#2\\
#3\\
\end{flushright}
\ifstrequal{#1}{.}{}{
\begin{center}
{\Large Homework #1}
\end{center}}
}

\newcommand{\seq}{\p}

\begin{document}
%
%\lstset{breaklines=true, morecomment=[l]{//}, frame=single, showstringspaces=false, numbers=left}
%
%\begin{flushright}
%Nathan Fox\\
%Math 640\\
%\end{flushright}
%\begin{center}
%\LARGE{On Aperiodic Subtraction Games with Bounded Nim Sequence}
%\end{center}
%\namehead{Math 640}{\LARGE{On Aperiodic Subtraction Games with Bounded Nim Sequence}}
\title{Quasipolynomial Solutions to the Hofstadter Q-Recurrence}
\author{Nathan Fox\footnote{Department of Mathematics, Rutgers University, Piscataway, New Jersey,
\texttt{fox@math.rutgers.edu}
}}
\date{}

\maketitle

\begin{abstract}
%In a recent paper, Frank Ruskey asked whether every linear recurrent sequence can occur in some solution of a meta-Fibonacci sequence.  In this paper, we answer his question in the affirmative for recurrences with positive coefficients.%superlinear recurrences.
In 1991, Solomon Golomb discovered a quasilinear solution to Hofstadter's $Q$-recurrence.  %In this paper, we give two constructions of eventual quasipolynomial solutions to meta-Fibonacci recurrences.  First, we show how to obtain eventual quasipolynomials of all positive degrees as solutions to a family of meta-Fibonacci recurrences.  Then, we show that we can also find eventual quasipolynomial solutions of every positive degree to Hofstadter's original recurrence, albeit with a larger blowup in quasiperiod.
In this paper, we construct eventual quasipolynomial solutions of all positive degrees to Hofstadter's recurrence.
\end{abstract}

\section{Introduction}
In the 1960s, Douglas Hofstadter introduced his $Q$ sequence~\cite[pp.\ 137-138]{geb}.  This sequence is defined by the recurrence $Q\p{n}=Q\p{n-Q\p{n-1}}+Q\p{n-Q\p{n-2}}$ along with the initial conditions $Q\p{1}=1$ and $Q\p{2}=1$.  Sequences defined in this way are often referred to as \emph{meta-Fibonacci sequences}~\cite{con}.  Though simple to define, this sequence appears to behave unpredictably.  To this day, it is even open whether this sequence is defined for all $n$.  It is conceivable that $Q\p{k}>k$ for some $k$, in which case $Q\p{k+1}$ would be undefined, as calculating it would refer to $Q$ of a nonpositive number.  Throughout this paper, though, we will use the convention that $Q\p{n}=0$ for $n\leq0$.  (We will call such an occurrence an \emph{underflow}.)  This may seem like cheating, but we could just as well replace the existence question about the $Q$ sequence by the equivalent question of whether $Q\p{n}\leq n$ for all $n$.  Other authors have also used this convention~\cite{rusk}.

In 1991, Golomb discovered a more predictable variation of Hofstadter's $Q$-sequence~\cite{golomb}.  He used the same recurrence, but, instead of the initial conditions $Q\p{1}=1$ and $Q\p{2}=1$, he used initial conditions $Q\p{1}=3$, $Q\p{2}=2$, and $Q\p{3}=1$.  This leads to a quasilinear sequence that can be described as follows:
\[
\begin{cases}
Q\p{3n}=3n-2\\
Q\p{3n+1}=3\\
Q\p{3n+2}=3n+2.
\end{cases}
\]
Given that one solution like this exists, it is conceivable that other related solutions exist.  In particular, under the aforementioned convention, it is plausible that Hofstadter's recurrence could have solutions where one equally-spaced subsequence grows quadratically.  This would occur if, for example, $Q\p{qn+r}$ equals $Q\p{qn+r-q}$ plus a linear polynomial, as the sequence $\seq{Q\p{qn+r}}_{n\geq1}$ would satisfy the recurrence $a_n=a_{n-1}+An+B$ for some $A$ and $B$.

In this paper, we show that quadratic solutions of this form do exist for Hofstadter's $Q$-recurrence.  In fact, we construct eventually-quasipolynomial solutions to the $Q$-recurrence of all positive degrees.
\section{The Construction}
First, we define the following:
\begin{defin}
Fix integers $d\geq1$ and $k\geq-1$.  Let
\[
p_{d,k}\p{n}=3d{{n+k}\choose{1+k}}+\sum_{i=1}^k\p{3i+2}{{n-1+k-i}\choose{k-i}}.
\]  
\end{defin}
Observe that $p_{d,k}$ is a polynomial in $n$ of degree $k+1$.  In particular, $p_{d,-1}=3d$, and $p_{d,0}=3dn$.  We will prove the following theorem:
\begin{theorem}\label{thm:main1}
Fix a degree $d\geq1$.  Define a sequence $\seq{a_m}_{m\geq1}$ as follows:
%\[
%a_n=\begin{cases}
%1 & n=d+1\\
%0 & n=d\\
%\p{d+2}n-1 & n\equiv d-1\p{\modd d+2}\\
%TODO
%\end{cases}
%\]
\[
a_{3dn+r}=\begin{cases}
3d-2 & 3dn+r=1\\
0 & 3dn+r=2\\
%3dn & r=0\\
p_{d,\frac{r}{3}}\p{n} & %r\equiv0\p{\modd3}\text{ and }r\neq0,
r\equiv0\p{\modd 3}\\
3d & r\equiv1\p{\modd 3}\text{ and }3dn+r>2\\
3 & r\equiv2\p{\modd 3}\text{ and }r\neq3d-1\text{ and }3dn+r>2\\
2 & r=3d-1\text{ and }3dn+r>2
%\displaystyle{3d{{n+\frac{r}{3}}\choose{1+\frac{r}{3}}}+\sum_{i=1}^{\frac{r}{3}}\p{3i+2}{{n+\frac{r}{3}-i-1}\choose{\frac{r}{3}-i}}} & r\equiv0\p{\modd3}\text{ and }r\neq0,
,
\end{cases}
\]
where $0\leq r<3d$ always.  Then, $\seq{a_m}$ satisfies the recurrence $Q\p{n}=Q\p{n-Q\p{n-1}}+Q\p{n-Q\p{n-2}}$ after an initial condition of length $3d+2$.
\end{theorem}
%Notice that the sequence $\seq{a_m}_{m\geq3}$ is a quasipolynomial of degree $d$, as the only exceptional indices are the first two.
%Notice that the expression corresponding to $r=3k$ is a polynomial of degree $k+1$.  Hence, the sequence $\seq{a_m}_{m\geq3}$ is a quasipolynomial of degree $d$, as the only exceptional indices are the first two.  Furthermore, observe that all coefficients in the polynomials are nonnegative,
%so the minimum value of the expression corresponding to $r$ is $r+2$ (corresponding to $i=\frac{r}{3}$ in the sum). 
%so the expression corresponding to $r$ is at least $\p{n+1}\p{r+2}$ (also since it is superlinear).
%We will use the shorthand $p_k$ to denote the polynomial when $r=3k$.
We will use the following lemmas:
\begin{lemma}\label{lem:ppascal}
For all integers $d\geq1$ and $k\geq0$ we have $p_{d,k}\p{n}=p_{d,k-1}\p{n}+p_{d,k}\p{n-1}$.
\end{lemma}
\begin{proof}
We have
\begin{align*}
p_{d,k-1}\p{n}+p_{d,k}\p{n-1}&=3d{{n+k-1}\choose{k}}+\sum_{i=1}^{k-1}\p{3i+2}{{n-2+k-i}\choose{k-i-1}}\\
&\hspace{0.2in}+3d{{n+k-1}\choose{1+k}}+\sum_{i=1}^k\p{3i+2}{{n-2+k-i}\choose{k-i}}\\
&=3d\p{{{n+k-1}\choose k}+{{n+k-1}\choose{1+k}}}\\
&\hspace{0.2in}+\sum_{i=1}^{k-1}\p{3i+2}\p{{{n-2+k-i}\choose{k-i-1}}+{{n-2+k-i}\choose{k-i}}}\\
&\hspace{0.2in}+\p{3k+2}{{n-2}\choose0}.
\end{align*}
Applying Pascal's Identity yields
\begin{align*}
p_{d,k-1}\p{n}+p_{d,k}\p{n-1}&=3d{{n+k}\choose{1+k}}+\sum_{i=1}^{k-1}\p{3i-2}{{n-1+k-i}\choose{k-i}}+\p{3k-2}\\
&=3d{{n+k}\choose{1+k}}+\sum_{i=1}^{k}\p{3i-2}{{n-1+k-i}\choose{k-i}}\\
&=p_{d,k}\p{n},
\end{align*}
as required.
\end{proof}
\begin{lemma}\label{lem:pg}
For all integers $d\geq1$, $k\geq1$, and $n\geq0$ we have
\[
p_{d,k}\p{n}\geq3dn+3k+2.
\]
\end{lemma}
\begin{proof}
First, we observe that
\[
p_{d,k}\p{0}=3d{k\choose{1+k}}+\sum_{i=1}^k\p{3i+2}{{k-i-1}\choose{k-i}}.
\]
All of these binomial coefficients are zero, except when $i=k$, since ${-1\choose0}=1$.  So, $p_{d,k}\p{0}=3k+2$.  This equals $3dn+3k+2$, and hence is greater than or equal to it, as required.

Now,
\begin{align*}
p_{d,k}\p{1}&=3d{{1+k}\choose{1+k}}+\sum_{i=1}^k\p{3i+2}{{k-i}\choose{k-i}}\\
&=3d+\sum_{i=1}^k\p{3i+2}\\
&=3d+3\p{\frac{k^2+k}{2}}+2k\\
&=\frac{3}{2}k^2+\frac{7}{2}k+3d.
\end{align*}
So,
\begin{align*}
p_{d,k}\p{1}-3d\cdot1+3k+2&=\frac{3}{2}k^2+\frac{7}{2}k+3d-3d-3k-2\\
&=\frac{3}{2}k^2+\frac{1}{2}k-2\\
&=\frac{\p{3k+4}\p{k-1}}{2}.
\end{align*}
This is greater than or equal to $0$, since $k\geq1$.  So, $p_{d,k}\p{1}\geq3d+3k+2$, as required.

Now, observe that $p_{d,k}$ has nonnegative coefficients, so it is convex.  We have seen that its average slope on the interval $\bk{0,1}$ is at least $3d$, so its derivative for $n>1$ must be strictly larger than $3d$ everywhere.  Therefore, since $p_{d,k}\p{1}\geq3d+3k+2$, we can conclude that $p_{d,k}\p{n}\geq3dn+3k+2$ for all $n\geq0$.
\end{proof}

We will now prove Theorem~\ref{thm:main1}.
\begin{proof}
We will check the three congruence classes mod $3$ separately for $m>3d+2$.  As usual, $m=3dn+r$ for $0\leq r<3d$.  We will proceed by induction, so in each case we will assume that all previous values of the sequence are what they should be.  Also, in all cases, since $m>3d+2$, $m-3d>2$.  (This will come up when deciding whether or not we are in the special initial conditions for the first two values.)
\begin{description}
%\item[$r=0$:] Assume $r=0$.  Then, $m=3dn$ for some $n$.  We have,
%\begin{align*}
%Q\p{3dn}&=Q\p{3dn-Q\p{3dn-1}}+Q\p{3dn-Q\p{3dn-2}}\\
%&=Q\p{3dn-Q\p{3d\p{n-1}+3d-1}}+Q\p{3dn-Q\p{3d\p{n-1}+3d-2}}\\
%&=Q\p{3dn-2}+Q\p{3dn-3d}\\
%&=Q\p{3d\p{n-1}+3d-2}+Q\p{3d\p{n-1}}\\%We don't go to the m=1 case
%&=3d+3d\p{n-1}\\
%&=3dn,
%\end{align*}
%as required.
\item[$r\equiv0\p{\modd3}$:] Assume $r\equiv0\p{\modd3}$.  Then, $m=3dn+r$ for some $n$.  For convenience, let $\ell=\frac{r}{3}$.  We wish to show that $Q\p{3dn+r}=p_{d,\ell}\p{n}$.  Let $c=2$ if $r=0$; otherwise, let $c=3$.  We have,
\begin{align*}
Q\p{3dn+r}&=Q\p{3dn+r-Q\p{3dn+r-1}}+Q\p{3dn+r-Q\p{3dn+r-2}}\\
&=Q\p{3dn+r-c}+Q\p{3dn+r-3d}\\
&=Q\p{3dn+r-c}+Q\p{3d\p{n-1}+r}\\
&=Q\p{3dn+r-c}+p_{d,\ell}\p{n-1}.
\end{align*}
If $r=0$, then $\ell=0$ and
\[
Q\p{3dn+r-c}=Q\p{3dn+r-2}=3d=p_{d,\ell-1}\p{n}.
\]
If $r\neq0$, then $\ell>0$ and
\[
Q\p{3dn+r-c}=Q\p{3dn+r-3}=p_{d,\ell-1}\p{n}.
\]
In either case, we have
\[
Q\p{3dn+r}=p_{d,\ell-1}\p{n}+p_{d,\ell}\p{n-1}.
\]
By Lemma~\ref{lem:ppascal}, this equals $p_{d,\ell}\p{n}$, as required.
\item[$r\equiv1\p{\modd3}$:] Assume $r\equiv1\p{\modd3}$.  Then, $m=3dn+r$ for some $n$.  We wish to show that $Q\p{3dn+r}=3d$. For convenience, let $\ell=\frac{r-1}{3}$. We have,
\begin{align*}
Q\p{3dn+r}&=Q\p{3dn+r-Q\p{3dn+r-1}}+Q\p{3dn+r-Q\p{3dn+r-2}}\\
&=Q\p{3dn+r-p_{d,\ell}\p{n}}+Q\p{3dn+r-Q\p{3dn+r-2}}.
\end{align*}
If $\ell=0$, then $p_{d,\ell}\p{n}=3dn$ and $r=1$.  So, in that case, $3dn+r-p_{d,\ell}\p{n}=r=1$.  Also, in that case $Q\p{3dn+r-2}=2$, so
\[
Q\p{3dn+r-Q\p{3dn+r-2}}=Q\p{3dn+r-2}=2.
\]
Since $Q\p{1}=3d-2$, we obtain $Q\p{3dn+r}=3d-2+2=3d$ in the case when $r=1$.  

Otherwise, we have $\ell\geq1$.  In that case,
%the first term becomes zero due to an underflow.  In particular,
%But, 
$p_{d,\ell}\p{n}\geq3dn+3\ell+2$ by Lemma~\ref{lem:pg}.  But, $3\ell+2=r+1$ so,
$
3dn+r-p_{d-1}\p{n}\leq-1.
$
This causes the first term to underflow,
%implies that
so $Q\p{3dn+r-p_{d,\ell}\p{n}}=0$.  Hence, $Q\p{3dn+r}=Q\p{3dn+r-Q\p{3dn+r-2}}$.  In this case, we know $r\neq1$, so $Q\p{3dn+r-2}=3$.  This means that
\[
Q\p{3dn+r-Q\p{3dn+r-2}}=Q\p{3dn+r-3}=3d.
\]
So, $Q\p{3dn+r}=3d$, as required.
\item[$r\equiv2\p{\modd3}$:] Assume $r\equiv2\p{\modd3}$.  Then, $m=3dn+r$ for some $n$.  Let $c=2$ if $r=3d-1$; otherwise, let $c=3$.  We wish to show that $Q\p{3dn+r}=c$. For convenience, let $\ell=\frac{r-2}{3}$. We have,
\begin{align*}
Q\p{3dn+r}&=Q\p{3dn+r-Q\p{3dn+r-1}}+Q\p{3dn+r-Q\p{3dn+r-2}}\\
&=Q\p{3dn+r-3d}+Q\p{3dn+r-p_{d,\ell}\p{n}}\\
&=Q\p{3d\p{n-1}+r}+Q\p{3dn+r-p_{d,\ell}\p{n}}\\
&=c+Q\p{3dn+r-p_{d,\ell}\p{n}}.
\end{align*}
If $\ell=0$, then $p_{d,\ell}\p{n}=3dn$ and $r=2$.  So, in that case, $3dn+r-p_{d,\ell}\p{n}=r=2$.  Since $Q\p{2}=0$, we obtain $Q\p{3dn+r}=c$ in the case when $r=2$.  

Otherwise, we have $\ell\geq1$.  In that case,
%But, 
$p_{d,\ell}\p{n}\geq3dn+3\ell+2$ by Lemma~\ref{lem:pg}.  But, $3\ell+2=r$ so,
$
3dn+r-p_{d-1}\p{n}\leq0,
$
an underflow in the second term.  This implies that $Q\p{3dn+r-p_{d,\ell}\p{n}}=0$, so $Q\p{3dn+r}=c$, as required.
\end{description}
\end{proof}
Note that the only place we used the $3i+2$ in the definition of $p_{d,k}\p{n}$ was to obtain the lower bound of $r+2$ on the polynomials.  So, $3i+2$ could be replaced by any larger expression, and the proof would still go through.  Also, observe that this construction is not a direct generalization of Golomb's construction, as the $d=1$ case has two constant pieces and one linear piece, unlike Golomb's, which has one constant piece and two linear pieces.  Also, Golomb's example is \emph{purely} quasilinear, whereas our $d=1$ example is only eventually quasilinear.  It is unknown whether there exist purely quasipolynomial solutions to the Hofstadter $Q$-recurrence of degrees greater than $1$.
\subsection{An Example}
As an example of our construction, we will construct a solution to Hofstadter's recurrence 
with a cubic subsequence.  To do this, we set $d=3$, which means that the sequence values will depend on the congruence class mod $9$ of the index.  We observe that
\begin{align*}
p_{3,0}&=9n\\
p_{3,1}&=9{{n+1}\choose2}+5{{n-1}\choose{0}}=\frac{9}{2}n\p{n+1}+5\\
&=\frac{9}{2}n^2+\frac{9}{2}n+5\\
p_{3,2}&=9{{n+2}\choose3}+5{{n}\choose{1}}+8{{n-1}\choose{0}}=\frac{9}{6}n\p{n+1}\p{n+2}+5n+8\\&=\frac{3}{2}n^3+\frac{9}{2}n^2+8n+8.
\end{align*}
So, our sequence is defined by
$a_1=7$, $a_2=0$, and for $9n+r>2$,
\[
a_{9n+r}=\begin{cases}
9n & r=0\\
9 &  r=1\\
3 &  r=2\\
\frac{9}{2}n^2+\frac{9}{2}n+5 &  r=3\\
9 &  r=4\\
3 &  r=5\\
\frac{3}{2}n^3+\frac{9}{2}n^2+8n+8 &  r=6\\
9 &  r=7\\
2 &  r=8.
\end{cases}
\]
%\[
%a_{9n+r}=\begin{cases}
%7 & 9n+r=1\\
%0 & 9n+r=2\\
%9n &  r=0\\
%\frac{9}{2}n^2+\frac{9}{2}n+5 & r=3\\
%\frac{3}{2}n^3+\frac{9}{2}n^2+8n+8 & r=6\\
%9 & r\equiv1\p{\modd3}\text{ and }9n+r>2\\
%3 & r\equiv2\p{\modd3}\text{ and }r\neq8\text{ and }9n+r>2\\
%2 & r=8.
%\end{cases}
%\]
After the initial condition $7,0,5,9,3,8,9,2,9,9,3$, repeated applications of the Hofstadter $Q$-recurrence produce the sequence
\[
7, 0, 5, 9, 3, 8, 9, 2, 9, 9, 3, 14, 9, 3, 22, 9, 2, 18, 9, 3, 32, 9, 3, 54, 9, 2, 27, 9, 3, 59, 9, 3, 113, 9, 2,\ldots
\]

\end{document}